\documentclass{amsart}
\usepackage{amsfonts,amsmath,amstext,amsthm,amssymb}
\usepackage{hyperref}
\numberwithin{equation}{section}
\newtheorem{theorem}{Theorem}
\newtheorem{OldTheorem}{Theorem}

\newtheorem{lemma}{Lemma}
\newtheorem{definition}{Definition}
\newtheorem*{remark}{Remark}
\newtheorem*{corollary}{Corollary}

\def\priority{{\rm \,priority\,}}
\def\mod{{\rm \,mod\,}}

\def\sp{{\rm sp\,}}
\def\tl{{\rm tl\,}}

\def\diam{{\rm diam\,}}

\def\bs{{\rm bs\,}}
\def\ZZ{\ensuremath{\mathbb Z}}
\def\ZI{\ensuremath{\mathbb I}}
\def\zI{\ensuremath{\mathcal I}}

\def\ZN{\ensuremath{\mathbb N}}
\def\ZT{\ensuremath{\mathbb T}}

\def\ZS{\ensuremath{\mathcal S}}
\def\zR{\ensuremath{\mathcal R}}
\def\ZM{\ensuremath{\mathcal M}}
\def\ZM{\ensuremath{\mathcal M}}
\def\ZR{\ensuremath{\mathbb R}}
\def\ZF{\ensuremath{\mathcal F}}
\def\ZA{\ensuremath{\mathcal A}}
\def\ZB{\ensuremath{\mathcal B}}

\def\md#1#2\emd{\ifx0#1
\begin{equation*} #2 \end{equation*}\fi  
\ifx1#1\begin{equation}#2\end{equation}\fi   
\ifx2#1\begin{align*}#2\end{align*}\fi   
\ifx3#1\begin{align}#2\end{align}\fi    
\ifx4#1\begin{gather*}#2\end{gather*}\fi  
\ifx5#1\begin{gather}#2\end{gather}\fi   
\ifx6#1\begin{multline*}#2\end{multline*}\fi  
\ifx7#1\begin{multline}#2\end{multline}\fi  
}
\newcommand {\e }[1]{(\ref{#1})}
\newcommand {\lem }[1]{Lemma \ref{#1}}
\newcommand {\trm }[1]{Theorem \ref{#1}}
\newcommand {\otrm }[1]{Theorem \ref{#1}}
\newcommand {\sect }[1]{Section \ref{#1}}

\begin{document}
\title{On Riemann sums and maximal functions in $\ZR^n$}

\author
{G. A. Karagulyan}

\address{Institute of Mathematics of Armenian
National Academy of Sciences, Bagramian Ave.- 24b, 375019,
Yerevan, Armenia}

\address{Yerevan State University, Depart. of Informatics and Applied Mathematics}

\email{g.karagulyan@yahoo.com}



\thanks{This research work was kindly
supported by College of Science-Research Center Project No.
Math/2008/07, Mathematics Department, College of Science, King
Saud University.}
\begin{abstract}
In this paper we investigate problems on almost everywhere
convergence of subsequences of Riemann sums
\begin{equation*}
R_nf(x)=\frac{1}{n}\sum_{k=0}^{n-1}f\bigg(x+\frac{k}{n}\bigg),\quad
x\in \ZT.
\end{equation*}
We establish a relevant connection between Riemann and ordinary
maximal functions, which allows to use techniques and results of
the theory of differentiations of integrals in $\ZR^n$ in
mentioned problems. In particular, we prove that for a definite
sequence of infinite dimension $n_k$ Riemann sums $R_{n_k}f(x)$
converge almost everywhere for any $f\in L^p$ with $p>1$.
\end{abstract}
\maketitle
\begin{section}{Introduction}

We consider the Riemann sums operators
\begin{equation*}
R_nf(x)=\frac{1}{n}\sum_{k=0}^{n-1}f\bigg(x+\frac{k}{n}\bigg),\quad
x\in \ZT,
\end{equation*}
for the functions defined on the torus $\ZT=[0,1]=\ZR/\ZZ $. It is
not hard to observe that if $f$ is continuous then these sums
converge to the integral of $f$ uniformly and they converge in
$L^1(\ZT)$ while $f$ is Lebesgue integrable. In this paper we
investigate certain problems concerning the almost everywhere
convergence of subsequences of Riemann operators. B.~Jessen's
classical theorem in \cite{Jess} is the first result in this
concern.
\begin{OldTheorem}[Jessen]
Let $\{n_k\}$ be an increasing sequence of positive integers such
that $n_k$ divides $n_{k+1}$. Then
\begin{equation}\label{Jesae}
\lim_{k\to\infty }R_{n_k}f(x)=\int_0^1f(t)dt,\hbox { a.e. }
\end{equation}
for any function $f\in L^1(\ZT)$. Moreover
\begin{equation}\label{Jes}
|\{x\in \ZT:\,\sup_kR_{n_k}|f(x)|>\lambda \}|\le\frac{1}{\lambda
}\|f\|_{L^1},\quad \lambda >0.
\end{equation}
\end{OldTheorem}
The next fundamental result in this direction due W. Rudin
\cite{Rud}. He has constructed an example of a bounded function
with divergent Riemann sums. Moreover it was proved
\begin{OldTheorem}[W. Rudin]\label{TRud}
Let $D$ be a sequence of positive integers which contains the sets
$D_n$ $(n=1,2,\ldots )$, each consisting of $n$ terms, such that
no member of $D_n$ divides the least common multiple of the other
members of $D_n$. Then for every $\varepsilon >0$ there exists a
bounded measurable function $f$, such that $0\le f\le 1$, and such
that
\begin{equation*}
\limsup_{n\to\infty,\, n\in D}R_nf(x)\ge \frac{1}{2}
\end{equation*}
for all $x$, although $\int f<\varepsilon $.
\end{OldTheorem}
For example, $D$ could be any sequence of primes. Using the
Dirichlet's theorem on primes in arithmetic progressions W. Rudin
in \cite{Rud} has constructed a sequence $\{n_k\}$ which satisfies
the hypothesis of Jessen's theorem such that $\{1+n_k\}$ is a
sequence of primes. Thus $R_{n_k}f(x)$ converges a.e., although
$R_{1+n_k}f(x)$ need not do so. This observation shows that in
a.e. convergence of operators $R_{n_k}f(x)$ arithmetic properties
of $\{n_k\}$ are crucial.

Following L. Dubins and J. Pitman \cite{DubPit}, we define a chain
to be an increasing sequence of natural numbers $\{n_k\}$ for
which $n_k$ divides $n_{k+1}$. For families of natural numbers
$\ZS_1,\ZS_2,\dots, \ZS_d $ we denote by $[\ZS_1, \ZS_2,\dots
,\ZS_d]$ the set of all naturals which are least common multiple
of some numbers $n_1\in \ZS_1, n_2\in\ZS_2, \dots , n_d\in\ZS_d$.
We will say a set $\ZS$ has dimension $d$, if $d$ is the least
possible integer such that $\ZS $ is the subset of $[\ZS_1,
\ZS_2,\dots ,\ZS_d]$ for some chains $\ZS_1,\ZS_2,\dots, \ZS_d $.
An example of a set of dimension $d$ is the set of integers having
the factorization
\begin{equation}\label{p1d}
p_1^{k_1}p_2^{k_2}\ldots p_d^{k_d},\quad k_1,k_2,\ldots ,k_d\in
\ZN,
\end{equation}
for fixed different primes $p_1,p_2,\ldots, p_d$. L. Dubins and J.
Pitman in \cite{DubPit} extended the Jessen's theorem proving
\begin{OldTheorem}\label{TDub}
If the set of positive integers has dimension $d$ and $f\in
L\log^{d-1} L(\ZT)$ then
\begin{equation*}
\lim_{n\to\infty , n\in\ZS }R_n(f)(x)=\int_0^1f(x)dx \hbox { a.e.
}
\end{equation*}
and moreover
\begin{equation}\label{DuPi}
m\{x\in \ZT:\sup_{n\in \ZS}|R_n(f)(x)|>\lambda \}\le
\frac{C_d}{\lambda}\int_0^1|f|\log^{d-1}(1+|f|).
\end{equation}
\end{OldTheorem}
In the original proof of this theorem the martingale theory was
used. There is a rather elementary and short proof of \e{DuPi}
given by an unknown referee of the article Y. Bugeaud and M. Weber
\cite {BugWeb}. More precisely, the maximal operator in \e{DuPi}
is estimated by $d$ iterations of the operator in \e{Jes}. Then
the inequality \e{DuPi} is derived by using an interpolation
theorem (\cite{Zyg2}, chap. 12, theorem 4.34). Another elementary
proof of this theorem has also suggested by R.~Nair in \cite
{Nair}. Y.~Bugeaud and M.~Weber in \cite{BugWeb} proved that
\otrm{TDub} is nearly sharp.
\begin{OldTheorem}\label{TBug}For any integer $d\ge 2$ and for any real
number $\varepsilon >0$ with $0 < \varepsilon < 1$, there exist a
sequence $n_k$ of dimension $d$ and a function $f\in
L\log^{d-1-\varepsilon }L(\ZT)$ such that $R_{n_k}f(x)$ is almost
everywhere divergent .
\end{OldTheorem}
The proof of this theorem is based on the method of R. C. Baker
\cite{Bak}, where author has proved a weaker version of this
theorem. As it is mentioned in \cite{BugWeb} \otrm{TBug} does not
answer precisely whether the class $L\log^{d-1} L(\ZT)$ in the
theorem is optimal or not. In \trm{T1} we prove that this class in
fact is exact and divergence can be everywhere.

In present paper we establish a direct connection between Riemann
maximal functions and ordinary maximal functions in Euclidian
spaces $\ZR^d$. Moreover it turns out, that Riemann maximal
function corresponding to a given finite set of indexes $D$ is
equivalent to a maximal function in Euclidian spaces $\ZR^d$ with
respect to certain $d$-dimensional rectangles which is the content
of \trm{T4} in \sect{S2}. \trm{T4} makes possible to use many
results and methods of maximal functions in this theory. Many
constructions used for Riemann sums get rather simple geometric
interpretation in $\ZR^d$. As applications of \trm{T4} we obtain
below solutions of some problems on Riemann sums. To figure out
the key point of our observation in \sect{S2} we display an
alternative proof to Jessen's theorem using a covering property of
some sets associated with Riemann sums. We will see a resemblance
between this proof and the proof of Hardy-Littlewood maximal
inequality where a covering lemma for intervals is used. In the
last section we deduce Rudin's theorem from \trm{T4} using a
simple geometry of multidimensional rectangles. In the same
section we prove that for a general class of operator sequences
the strong sweeping out and $\delta $-sweeping out properties are
equivalent.

Let $\Phi:\ZR^+\to\ZR^+$ be an increasing convex function. Denote
by $L^\Phi (\ZT)$ the class of functions $f$ on $\ZT$ with
$\Phi(|f|)\in L^1(\ZT)$. If $\Phi $ satisfies $\Delta_2$-condition
$\Phi(2x)\le k\Phi(x)$ then $L^\Phi$ is Banach space with the norm
$\|f\|_{L^\Phi}=\|f\|_\Phi$ to be the least $c>0$ for which the
inequality
\begin{equation*}
\int_\ZT\Phi\left( \frac{|f|}{c}\right)\le 1
\end{equation*}
holds. The following theorem makes correction in the last theorem
and shows that the class $L\log^{d-1} L$ in Theorem C is exact.
\begin{theorem}\label{T1}
Let $n_k$ be the increasing sequence formed the numbers \e{p1d}
with fixed different primes $p_1,p_2,\ldots, p_d$. If an
increasing function $\phi :\ZR^+\to\ZR^+$ satisfies the condition
\begin{equation*}
\lim_{x\to\infty }\frac{\phi(x)}{x\ln^{d-1}x}=0,
\end{equation*}
then there exists a function $f(x)\in L^\phi$ such that the
sequence $R_{n_k}f(x)$ is everywhere divergent.
\end{theorem}
According to the \otrm{TDub}, Riemann sums corresponding to a set
of finite dimension converge a.e. in $L^p$ classes with $p>1$. As
for the sets of infinite dimension it was a problem wether there
exists a sequence of infinite dimension $\{n_k\}$ such that
$R_{n_k}f(x)$ converges for any function $f\in L^p(\ZT)$ with
$p>1$. In \cite{BugWeb} Y.~Bugeaud and M.~Weber discussed a
particular sequence of infinite dimension $E$ consist of all
integers defined
\begin{equation}\label{E1}
E=\{p_1\ldots p_{j-1}\check{p}_j p_{j+1}\ldots p_k:\quad
k=2,3,\ldots ,\,1\le j\le k\}
\end{equation}
where $p_1<p_2<\ldots $ is the sequence of primes and the symbol
$\check{}$ means $p_j$ must be excluded in the product. As it is
proved in \cite{DubPit} $E$ has infinite dimension. In
\cite{BugWeb} (see also \cite{RuWe}) it is proved the almost
everywhere convergence of Riemann sums $R_{n_k}f(x)$ where
$\{n_k\}=E$ for the functions $f\in L^2(\ZT)$ with Fourier
coefficients satisfying
\begin{equation*}
\sum_{n>3}a_n^2\left(\frac{\ln n}{\ln\ln n}\right)<\infty .
\end{equation*}
It is proved also
\begin{equation*}
\lim_{N\to\infty
}\frac{1}{N}\sum_{k=1}^NR_{n_k}f(x)=\int_0^1f(x)dx\hbox { a.e. }
\end{equation*}
for any $f\in L^2(\ZT)$. We proved that Riemann sums associated to
the set $E$ converge a.e. in any $L^p$, $p>1$. Moreover, a.e.
convergence  holds in the Orlicz class $L^\Phi$ corresponding to
the function
\begin{equation}\label{phi2}
\Phi(x)=\frac{x\ln (1+x)}{\ln\ln (3+x)},\quad x\ge 0,
\end{equation}
and this class is the optimal one for the set $E$. So we prove the
following theorems.
\begin{theorem}\label{T2}
Let $E$ be the set defined in \e{E1} and $\Phi (x) $ is the
function \e{phi2}. Then for any $f\in L^\Phi$ we have
\begin{equation*}
\lim_{n\to \infty,\,n\in E }R_nf(x)=\int_0^1f\hbox { a.e. }.
\end{equation*}
Moreover
\begin{equation}\label{T2form}
|\{x\in \ZT:\sup_{n\in E}|R_nf(x)|>\lambda \}|\le
\int_0^1\Phi\left(\frac{c|f|}{\lambda}\right),\quad \lambda >0 ,
\end{equation}
where $c>0$ is an absolute constant.
\end{theorem}
This theorem immediately implies
\begin{corollary}
There exists an infinite set $E\subset \ZN$ such that for any
$f\in L^p(\ZT)$ with $p>1$ Riemann sums $R_nf(x)$, $n\in E$
converge a.e.
\end{corollary}
\begin{theorem}\label{T3}
If the sequence $n_1<n_2<\ldots $ consists of all the integers of
the set $E$ and the increasing function $\phi :\ZR^+\to\ZR^+$
satisfies the condition
\begin{equation}\label{phis}
\lim_{x\to\infty}\frac{\phi(x)\ln\ln x}{x\ln x}=0,
\end{equation}
then there exists a function $f(x)\in L^\phi $ such that the
sequence $R_{n_k}f(x)$ is everywhere divergent.
\end{theorem}
\end{section}

\begin{section}{Notations}
We recall some definitions in measure theory (see \cite{Hal}). Let
$X$ to be an arbitrary set. A family $\Omega $ of subsets of $X$
is called algebra if it is closed with respect to the operations
of union, intersection and difference and $X\in\Omega $. If the
algebra is closed also with respect to countable union it is
called $\sigma$-algebra. The set $A$ is called atom for the
algebra $\Omega $ if there is no nonempty $B\in \Omega $ so that
$B\subset A$. We note that if the algebra $\Omega $ is finite then
any set from $\Omega $ is a union of some atoms of $\Omega $. If
there is also a measure $\mu$ on $\Omega $ we denote this measure
space by $(X,\Omega,\mu)$. It is said the measure spaces
$(X,\Omega,\mu)$ and $(Y,\Delta,\nu)$ are isomorph if there exists
a one to one mapping $\gamma :\Omega\to\Delta$ called isomorphism
such that
\begin{equation*}
\gamma(A-B)=\gamma(A)-\gamma(B),\quad
\gamma\left(\bigcup_{k=1}^\infty A_k\right)= \bigcup_{k=1}^\infty
\gamma(A_k),
\end{equation*}
and
\begin{equation*}
 \nu(\gamma (A))=\mu(A),
\end{equation*}
for any sets $A,B$ and $A_k$, $k=1,2,\ldots $, from $\Omega$. If
$\Omega$ is not $\sigma $-algebra we suppose in addition
$\cup_{k=1}^\infty A_k\in \Omega$.  We will say $f:\,X\to Y$ is an
isomorphism function if the set function $\gamma(A)=\{y\in Y:\,
y=f(x),\, x\in A\}$ determines one to one mapping between $\Omega
$ and $\Delta $ which is an isomorphism. Suppose the algebras
$\Omega$ and $\Delta$ are finite and have atoms $A_1,A_2,\ldots
,A_n$ and $B_1,B_2,\dots ,B_n$ respectively. It is clear if
$\nu(A_i)=\mu(B_i)$, $i=1,2,\ldots ,n$, then the measure spaces
$(\Omega,\mu )$ and $(\Delta,\nu )$ are isomorph.

We consider the probability space $(\ZT^\infty,\lambda)
=\prod_{i=1}^\infty (\ZT_i,\lambda_i)$ where each
$(\ZT_i,\lambda_i)$ is the Lebesgue probability space on $\ZT$.
Remind that measurable sets in $\ZT^\infty $ is generated from all
the products $A=\prod_{i=1}^\infty A_i$ where each $A_i$ is
Lebesgue measurable set in $\ZT$ and only finite number of them
differ from $\ZT$ (see definition in \cite{Nev}, chap. III.3). The
measure in $\ZT^\infty$ is the extension of the measure
$\lambda(A)=\prod_{i=1}^\infty |A_i|$. We will use $|A|$ to
indicate measure of $A\subset \ZT^\infty $.

Let $l\in \ZN$ and $D\subset\ZN $ is finite. We will write $D|l$
if any member of $D$ divides $l$. We denote
\begin{equation}\label{l/D}
l/D=\{n\in\ZN:\, \frac{l}{n}\in D\}.
\end{equation}

An important subject in this paper is the relationship between
three type of sets. Namely we will consider Riemann sets, integer
arithmetic progressions and special rectangles in $\ZT^\infty$
having the following descriptions.

 \textit{Riemann sets:} We denote by $\zI_l$ the algebra in $\ZT$
generated by intervals $[\frac{j}{l},\frac{j+1}{l})$,
$j=0,1,\ldots l-1$. Define Riemann sets
\begin{equation}\label{form}
I_l(n,t)=\bigcup_{i=0}^{n-1}\bigg[\frac{t}{l}+\frac{i}{n},
\frac{t+1}{l}+\frac{i}{n}\bigg),\quad
t=0,1,\ldots ,\frac{l}{n}.
\end{equation}
where $n$ divides $l$. Certainly we have
\begin{equation*}
I_l(n,t)\in \zI_l,\quad \lambda(I_l(n,t))=\frac{n}{l}.
\end{equation*}
For fixed $l$ and $n$ dividing $l$ the collection \e{form} is a
pairwise disjoint partition of $[0,1]$. It is easy to verify if
$x\in [k/l,(k+1)/l)$ then
\begin{equation*}
l\int_{k/l}^{(k+1)/l}R_nf(t)dt=\frac{1}{|I_l(n,t)|}\int_{I_l(n,t)}
f(t)dt,\quad x\in I_l(n,t).
\end{equation*}
Thus, using Lebesgue's theorem on $\ZR$, we get
\begin{equation}\label{limRm}
\lim_{l\to \infty,\,x\in
I_l(n,t)}\frac{1}{|I_l(n,t)|}\int_{I_l(n,t)}f(t)dt=R_nf(x)\hbox{
a.e. }, \quad n=1,2,\cdots .
\end{equation}
For any subset $D\subset \ZN$ we define
\begin{equation}\label{ZM}
\zR_D^lf(x)=\sup_{n\in D:\, n|l,\,x\in I_l(n,t)}
\frac{1}{|I_l(n,t)|}\int_{I_l(n,t)} |f(t)|dt.
\end{equation}
If $D_n$ are finite subsets of $D$ with $\cup_n D_n=D$ then from
\e{limRm} it follows that
\begin{equation}\label{RDMD}
\zR_Df(x)=\sup_{n\in D}R_nf(x)=\lim_{n\to \infty }\lim_{l\to
\infty :\, D_n|l}\zR_D^lf(x) \hbox { a.e. }.
\end{equation}
Observe that if $f(x)$ is $\zI_l$-measurable then
\begin{equation}\label{DDl}
\zR_Df(x)=\zR_D^lf(x).
\end{equation}
Indeed, since from $l|l'$ follows $\ZA_l\subset \ZA_{l'}$ we
derive
\begin{equation*}
\zR_D^{l'}f(x)=\zR_D^lf(x)\hbox { if } l|l',
\end{equation*}
and therefore by \e{RDMD} we have
\begin{equation*}
\zR_Df(x)=\lim_{l'\to\infty :\,l|l' }\zR_D^{l'}f(x)=\zR_D^lf(x).
\end{equation*}

\textit{Arithmetic progressions:} We shall say a set of integers
$A$ is $l$-periodic if $A=l+A$. We denote by $\ZA_l$ the family of
all $l$-periodic sets of integers and $\ZA=\cup_{l\in\ZN}\ZA_l$.
It is clear if $l|l'$ then any $l$-periodic set is $l'$-periodic,
i.e. $\ZA_l\subset \ZA_{l'}$. Observe that $\ZA$ and each $\ZA_l$
are algebras. We define the measure of a set $A\in\ZA $ by
\begin{equation*}
\delta(A)=\lim_{l\to\infty}\frac{\#(A\cap [0,l))}{l},
\end{equation*}
where $\#B$ denotes the cardinality of the finite set $B$. It is
clear that the limit exists and if $A\in\ZA_l$ then
\begin{equation*}
\delta(A)=\frac{\#(A\cap [0,l))}{l}.
\end{equation*}
Observe that $\delta$ is an additive measure on $\ZA$. Now
consider the arithmetic progressions
\begin{equation}\label{Al}
A_l(t)=\{ lj+t,\, j\in \ZZ\},\quad 0\le t < l.
\end{equation}
It is clear
\begin{equation*}
A_l(t)\in \ZA_l,\quad \delta\big(A_l(t)\big)=\frac{1}{l},
\end{equation*}
and any set from $\ZA_l$ can be written as a finite union of these
arithmetic progressions. It means the sets in \e{Al} are the atoms
of the algebra $\ZA_l$.

\textit{Rectangles in $\ZT^\infty$:} We denote $p_1<p_2\ldots
<p_d<\ldots $ the sequence of all primes. Consider an integer $l$
with factorization
\begin{equation}\label{l}
l=p_1^{l_1} p_2^{l_2}\cdots p_d^{l_d}.
\end{equation}
We do not exclude that some of the numbers $l_k$ are zero. Define
rectangles in $\ZT^\infty $ by
\begin{equation}\label{Rl}
B_l(j_1,\cdots , j_d)= \left\{x\in\ZT^\infty :\,
\frac{j_k}{p_k^{l_k}}\le x_k<\frac{j_k+1}{p_k^{l_k}},\,
k=1,2,\ldots ,d\right\},
\end{equation}
where
\begin{equation*}
0\le j_k<p_k^{l_k},\quad k=1,2,\cdots ,d,\quad
x=(x_1,x_2,\ldots,x_k,\ldots ).
\end{equation*}
We denote by $\ZB_l$ the algebra generated of all the finite
unions of the rectangles \e{Rl}. So the family
\begin{equation*}
\ZB=\bigcup_{l\in \ZN}\ZB_l.
\end{equation*}
is an algebra in $\ZT^\infty $. We note that $\ZB_l\subset
\ZB_{l'}$ while $l|l'$. We shall consider the measure space $(\ZB
,\lambda )$, where $\lambda $ is the Lebesgue's measure on
$\ZT^\infty $. It is clear
\begin{equation*}
\lambda(B_l(j_1,\cdots , j_d))=\frac{1}{l}.
\end{equation*}

It is clear that $(\zI_l,\lambda )$, $(\ZA_l,\delta )$ and
$(\ZB_l,\lambda )$ are isomorph, because all have $l$ atoms with
with equal measures. In \sect{S3} we are going to construct a
special isomorphism between $\ZA$ and $\ZB$ assigning the
arithmetic progressions \e{Al} to the rectangles \e{Rl}.
\end{section}

\begin{section}{An alternative proof of Jessen's theorem}\label{S2}

Operators \e{RDMD} play a significant role in the study of a.e.
convergence of Riemann sums. To prove Jessen's theorem it is
enough to prove the inequality \e{Jes}, because \e{Jesae} follows
from \e{Jes} by using Banach principle. So we suppose
$D=\{m_1,m_2,\cdots, m_d ,\ldots\}$ where $m_k$ divides $m_{k+1}$.
We fix a finite subset $U=\{m_1,m_2,\cdots, m_d\}\subset D$ and an
integer $l$ divided by $m_d$ and so all $m_k$, $1\le k\le d$. It
is clear
\begin{equation*}
\{\zR_U^lf(x)>\lambda \}=\bigcup_jI_j,
\end{equation*}
where $I_j$ are Riemann sets form $\zI_l$ with
\begin{equation*}
 \frac{1}{|{I_j}|}\int_{I_j}
|f(t)|dt>\lambda .
\end{equation*}
We will prove that it may be chosen a subfamily of mutually
disjoint sets $\{\tilde I_j\}$ such that
\begin{equation}\label{Itil}
\bigcup_jI_j=\bigcup_j\tilde I_j.
\end{equation}
 We define $\priority (I)$=n if $I$ has
the form \e{form}. It is easy to observe that if $\priority (I)$
divides $\priority (J)$ and $I\cap J\neq\varnothing $ then we have
$I\subseteq J $. We take $\tilde I_1$ to be some of $I_j $ with
highest priority. Suppose we have chosen $\tilde I_1, \tilde I_2,
\cdots ,\tilde I_m$. We consider all $I_j$'s with
$I_j\nsubseteq\bigcup_{j=1}^m\tilde I_j $ and so
$I_j\cap\left(\cup_{j=1}^m\tilde I_j \right)=\varnothing $. We
take $\tilde I_{m+1}$ among these sets having an highest priority.
Certainly this process generates a subcollection $\{\tilde I_j\}$
of mutually disjoint sets with \e{Itil}. Thus we obtain
\begin{equation*}
|\{\zR_B^lf(x)>\lambda \}|=\big|\bigcup_j I_j\big|=\sum_j|\tilde
I_j| <\frac{1}{\lambda} \sum_j \int_{\tilde I_j} |f(t)|dt\le
\frac{\|f\|_{L^1}}{\lambda}.
\end{equation*}
Since the inequality is true for any finite $U\subset D$, applying
\e{RDMD} we get \e{Jes}.
\end{section}


\begin{section}{An isomorphism between arithmetic progressions and rectangles}\label{S3}
Let $l$ be an integer with factorization \e{l} and
\begin{equation}\label{m}
m=p_1^{m_1} p_2^{m_2}\cdots p_d^{m_d},\quad 0\le m_k\le l_k,\quad
k=1,2,\cdots ,d.
\end{equation}
From the definition \e{Rl} it follows that
\begin{equation}\label{RmRl}
B_m(t_1,\ldots ,t_d)=\bigcup_{p_k^{l_k-m_k} t_k\le
s_k<p_k^{l_k-m_k}\big( t_k+1\big)} B_l( s_1,\cdots , s_d).
\end{equation}
For a fixed integer $t$ we consider the set of integer vectors
\begin{equation}\label{sk}
S_t=\{(s_1,s_2,\cdots ,s_d):\,0\le s_k<p_k^{l_k},\quad
s_k=t\mod p_k^{m_k},\quad k=1,2,\cdots ,d,\}
\end{equation}
In fact $S_t$ depends also on $l$ and $m$.
\begin{lemma}
There exists a one to one correspondence from
\begin{equation}\label{uset}
U=\{0,1,\cdots ,\frac{l}{m}-1\}
\end{equation}
to the set of vectors \e{sk}, such that the vector
$(s_1,s_2,\cdots ,s_d)\in S_t$ assigned to $u\in U$ satisfies
\begin{equation}\label{mut}
s_k=(mu+t)\mod p_k^{l_k},\quad k=1,2,\cdots ,d.
\end{equation}
\end{lemma}
\begin{proof} We note that there are $p_k^{l_k-m_k}$ number of
$s_k$'s satisfying
\begin{equation*}
0\le s_k<p_k^{l_k},\quad s_k=t\mod p_k^{m_k}.
\end{equation*}
So we have
\begin{equation*}
\#S_t=\prod_{k=1}^dp_k^{l_k-m_k}=\frac{l}{m}=\#U.
\end{equation*}
Thus, it is enough to prove that for any $u\in U$ there exists a
vector $(s_1,\cdots ,s_d)\in S$ with \e{mut}, and the images of
different $u$'s are different. To determine the vector
$(s_1,\ldots ,s_d)$ corresponding to $u$ we define $s_k$ to be the
remainder when $mu+t$ is divided by $p_k^{l_k}$. Certainly
$(s_1,\ldots ,s_d)$ satisfies \e{mut} and $0\le s_k<p_k^{l_k}$.
Since $p_k^{m_k}|m$ we get $s_k=t\mod p_k^{m_k}$. So $(s_1,\ldots
,s_d)\in S_t$. Now we suppose $(s_1,s_2,\cdots ,s_d)\in S_t$ is
assigned to different $u_1,u_2\in U$. The numbers $u_1$ and $u_2$
satisfy the relation \e{mut}. Hence, we get
\begin{equation*}
m(u_1-u_2)=0\mod p_k^{l_k},\quad k=1,2,\cdots ,d.
\end{equation*}
Since $0\le u_1-u_2<\frac{l}{m}$, using \e{l} and \e{m}, we
conclude $u_1-u_2=0$.
\end{proof}
\noindent
Let $p\ge 2$ be an integer. Any nonnegative integer $a$ has
$p$-adic decomposition
\begin{equation*}
a=a_0p^k+a_1p^{k-1}+\cdots +a_k,\quad 0\le a_i<p.
\end{equation*}
We denote
\begin{equation*}
(a)_p=a_kp^k+a_{k-1}p^{k-1}+\cdots +a_0,
\end{equation*}
the integer with revers arrangement of $p$-digits of $a$. We shall
say that $(a)_p$ is the $p$-reverse of $a$. We note this action
defines a one to one mapping of the set of integers $\{0,1,\cdots
,p^k-1\}$ into itself. Notice if
\begin{equation*}
s=p^iv+t,\quad 0\le v<p^{j-i},\quad 0\le t<p^i,\,i\le j,
\end{equation*}
then
\begin{equation}\label{sbar}
\bar s=p^{j-i}\bar t+\bar v.
\end{equation}
It is easy to observe that for a fixed $t$ the correspondence
$s\to \bar s$ is a one to one mapping between the sets
\begin{equation*}
\{s:\, s=p^iv+t,\,0\le v<p^{j-i}\}\hbox{ and } \{s:\,
s=p^{j-i}\bar t+v,\,0\le v<p^{j-i}\}.
\end{equation*}
\begin{lemma}\label{L1}
There exists an isomorphism $\alpha $ from the measure space
$(\ZA,\delta )$ to $(\ZB,\lambda )$ assigning any progression
\e{Al} to a rectangle \e{Rl}.
\end{lemma}
\begin{proof}
At first we define $\alpha $ on the progressions \e{Al}. We take
an arbitrary $A_l(t)$. Suppose
\begin{equation*}
t_k=t\mod p_k^{l_k}, \quad 0\le t_k< p_k^{l_k},\quad k=1,2,\ldots
,d,
\end{equation*}
and denote by $\bar t_k$ the $p_k$-reverse of the integer $t_k$.
We have $0\le \bar t_k< p_k^{l_k}$. We define
\begin{equation}\label{Phidef}
\alpha\big(A_l(t)\big)=B_l(\bar t_1,\cdots , \bar t_d).
\end{equation}
According the definition \e{Al} for a given arithmetic progression
$A_m(t)$ we have
\begin{equation}\label{Amt}
A_m(t)=\bigcup_{u=0}^{l/m-1}A_l(mu+t).
\end{equation}
We shall prove that
\begin{equation}\label{Phieq}
\alpha\big(A_m(t)\big)=\bigcup_{u=0}^{l/m-1}\alpha\big(A_l(mu+t)\big).
\end{equation}
According to \lem{L1} there exists a one to one mapping between
the sets $U$ and $S_t$ defined in \e{sk} and \e{uset}. In
addition, if $(s_1,s_2,\cdots ,s_d)\in S_t$ is assigned to a given
$u\in U$ then it satisfies the condition \e{mut} and therefore by
\e{Phidef} we have
\begin{equation}\label{Amut}
\alpha(A_l(mu+t))=B_l(\bar s_1,\ldots ,\bar s_d).
\end{equation}
Now let $t_k$ be the remainder when $t$ is divided by $p_k^{m_k}$,
i.e.
\begin{equation}\label{tkt}
t_k=t \mod p_k^{m_k}, \quad 0\le t_k<p_k^{m_k}.
\end{equation}
From \e{Phidef} we get
\begin{equation*}
\alpha\big(A_m(t)\big)=B_m(\bar t_1,\cdots , \bar t_d).
\end{equation*}
From \e{sk} it follows that $p^{m_k}$ divides $s_k-t$ and
therefore by \e{tkt} it divides also $s_k-t_k$. So we have
\begin{equation*}
s_k=p_k^{m_k}v_k+t_k,\quad 0\le v_k<p_k^{l_k-m_k},\quad 0\le
t_k<p_k^{m_k}.
\end{equation*}
Thus, according to \e{sbar}, for the $p_k$-revers $\bar s_k$ of
the integer $s_k$ we have
\begin{equation*}
\bar s_k=p_k^{l_k-m_k}\bar t_k+\bar v_k,\quad  0\le \bar
v_k<p_k^{l_k-m_k},\quad 0\le \bar t_k<p_k^{m_k},
\end{equation*}
where $\bar v_k$ and $\bar t_k$ are the $p_k$-reverses of $v_k$
and $t_k$ respectively. Hence for any  $u\in U$ may be determined
$(\bar s_1,\cdots ,\bar s_d)$ with
\begin{equation}\label{barvk}
p_k^{l_k-m_k}\bar t_k\le\bar s_k<p_k^{l_k-m_k}\big(\bar
t_k+1\big).
\end{equation}
In addition, it is easy to check this correspondence is a one to
one mapping from $U$ to the set of vectors $(\bar s_1,\bar
s_2,\cdots ,\bar s_d)$ with \e{barvk}. Therefore, according to
\e{Amt}, \e{Amut} and \e{RmRl}, we get
\begin{multline*}
\alpha\big(A_m(t)\big)=B_m(\bar t_1,\cdots , \bar
t_d)\\
=\bigcup_{p_k^{l_k-m_k}\bar t_k\le  \bar
s_k<p_k^{l_k-m_k}\big(\bar t_k+1\big)} B_l( \bar s_1,\cdots , \bar
s_d)=\bigcup_{u=0}^{l/m-1}\alpha\big(A_l(mu+t)\big).
\end{multline*}
So \e{Phieq} is true. Now take an arbitrary set $A\in \ZA$. We
have $A\in \ZA_l$ for some $l\in \ZN$. Since \e{Al} are the atoms
of $\ZA_l$, the set $A$ is a union of some mutually disjoint
atoms, i.e.
\begin{equation}\label{Arep}
A=\bigcup_{i\in I}A_l(i).
\end{equation}
We define
\begin{equation}\label{PArep}
\alpha(A)=\bigcup_{i\in I}\alpha(A_l(i)).
\end{equation}
Since $A$ belongs to different algebras $\ZA_l$, there are
different representations \e{Arep} corresponding to different
$l$'s. However, using \e{Phieq}, it is easy to verify that the
right side of \e{PArep} does not depend on the representation
\e{Arep}. On the other hand $\alpha $ is measure preserving,
because $\delta(A_l(i))=\lambda(\alpha(A_l(i)))=1/l$ by
\e{Phidef}. So we conclude that $\alpha$ is an isomorphism from
$\ZA$ to $\ZB$. In addition, according to \e{Phidef} it assigns
any progression \e{Al} to a rectangle \e{Rl}. The proof of
\lem{L1} is complete.
\end{proof}
For any $l$-periodic set of integers $A\in \ZA_l$ we define
\begin{equation*}
\beta_l(A)=\bigcup_{k\in A}\left[\frac{k}{l},\frac{k+1}{l}\right).
\end{equation*}
It is easy to check that $\beta_l$ determines an isomorphism from
the probability space $(\ZA_l,\lambda )$ to $(\zI_l,\delta )$.
Moreover
\begin{equation*}
\beta((A_{l/n}(t))=I_l(n,t).
\end{equation*}
Thus, the composition of $\alpha\circ\beta_l^{-1}$ where $\alpha $
is from \lem{L1} is an isomorphism from $(\zI_l,\lambda )$ to
$(\ZB_l,\lambda )$. Moreover the following lemma is true.
\begin{lemma}\label{L3}
For any $l\in \ZN$ there exists a one to one mapping
$\tau_l:\ZT\to \ZT^\infty$ such that
\begin{enumerate}
    \item $\tau_l$ is measure preserving, i.e. $|\tau(A)|=|A|$ for any Lebesgue
measurable $A\subset\ZT$,
    \item $\tau_l$ is an isomorphism function between $(\ZT,\zI_l,\lambda)$ and
 $(\ZT^\infty,\ZB_l,\lambda )$
    \item for any $I_l(n,t)$ from \e{form} the set $\gamma_l(I_l(n,t))$
is a rectangle of the form $B_m(i_1,\ldots ,j_d)$ with
$m=\frac{l}{n}$.
\end{enumerate}
\end{lemma}
\begin{remark}
The existence of a mapping with the conditions (1) and (2) is
trivial. The important part of the lemma is the fact that
$\gamma_l(I_l(n,t))$ is a certain rectangle in $\ZT^\infty$.
\end{remark}
For any set of integers $D\subset \ZN$ we define the maximal
function
\begin{equation}\label{MBl}
\ZM_Dg(x)=\sup_{m\in D:\,x\in B_m(j_1,\ldots ,j_d)}
\frac{1}{|B_m(j_1,\ldots ,j_d)|}\int_{B_m(j_1,\ldots ,j_d)}
|g(t)|dt
\end{equation}
where $g\in L^1(\ZT^\infty)$. We note that if $l$ is a multiple
for the numbers from $D$ then the rectangles in \e{MBl} are in
${\ZB_l}$. This implies that for the conditional expectation
$E^{\ZB_l}g(x)$ of $g(x)$ with respect to the algebra ${\ZB_l}$ we
have
\begin{equation}\label{cexp}
\ZM_Dg(x)=\ZM_DE^{\ZB_l}g(x).
\end{equation}

The following theorem clearly follows from \lem{L3}. It creates an
equivalency between Riemann maximal function $\zR_D^lf(x)$ defined
in \e{ZM} and $\ZM_{l/D}g(x)$, where $l/D$ is defined in \e{l/D}.
\begin{theorem}\label{T4}
For any $l\in \ZN$ there exists a measure preserving mapping
$\tau_l:\ZT\to \ZT^\infty$ such that if $f(x)\in L^1(\ZT)$ and
$g(x)=f(\tau_l^{-1}(x))$ then
\begin{equation*}
|\{x\in \ZT:\,\zR_D^lf(x)>\lambda \}|=|\{x\in
\ZT^\infty:\,\ZM_{l/D}g(x)>\lambda \}|,\quad \lambda>0.
\end{equation*}
\end{theorem}

\begin{corollary}
Let $D$ be a set of indexes and $\Phi:\ZR^+\to\ZR^+$ be an
increasing convex function. Then
\begin{equation}\label{cor}
\sup_{\|f\|_\Phi\le 1}\bigg|\bigg\{x\in \ZT:\, \zR_Df(x)>\lambda
\bigg\}\bigg|=\sup_{B\subset D,\, l\in \ZN,\,\|g\|_\Phi\le
1}\bigg|\bigg\{x\in \ZT^\infty:\, \ZM_{l/B}g(x)>\lambda
\bigg\}\bigg|,
\end{equation}
for any $\lambda >0$, where in $\sup $ finite sets $B$ are
considered.
\end{corollary}
\begin{proof}
Take $f\in L^\Phi(\ZT)$. If $\tau_l$ is the mapping satisfying the
conditions of \trm{T4} then the functions
$g_l(x)=f(\tau_l^{-1}(x))$, $l=1,2,\ldots $, satisfy
\begin{equation*}
|\{x\in \ZT:\, \zR_B^lf(x)>\lambda\}| =\{x\in \ZT^\infty:\,
\ZM_{l/B}g_l(x)>\lambda \},
\end{equation*}
and $\|f\|_\Phi=\|g_l\|_\Phi $ since $\tau_l$ is measure
preserving. Taking into account \e{RDMD} we obtain
\begin{multline*}
\bigg|\bigg\{x\in \ZT:\, \zR_Df(x)>\lambda \bigg\}\bigg|\\
\le \sup_{B\subset D,\, l\in \ZN}\bigg|\bigg\{x\in \ZT:\,
\zR_B^lf(x)>\lambda \bigg\}\bigg| =\sup_{B\subset D,\, l\in
\ZN}\bigg|\bigg\{x\in \ZT^\infty:\, \ZM_{l/B}g_l(x)>\lambda
\bigg\}\bigg|.
\end{multline*}
Since $f\in L^\Phi$ is arbitrary
and $\|f\|_\Phi=\|g_l\|_\Phi $ we get
\begin{equation*}
\sup_{\|f\|_\Phi\le 1}\bigg|\bigg\{x\in \ZT:\, \zR_Df(x)>\lambda
\bigg\}\bigg|\le\sup_{B\subset D,\, l\in \ZN,\,\|g\|_\Phi\le
1}\bigg|\bigg\{x\in \ZT^\infty:\, \ZM_{l/B}g(x)>\lambda
\bigg\}\bigg|.
\end{equation*}
Now suppose $g\in L^\Phi(\ZT^\infty )$, $B\subset D$ is finite and
$l\ge 2$ is arbitrary integer. According to \e{cexp} there exists
$\ZA_l$-measurable function $g_l$ such that
\begin{equation}\label{gf1}
\ZM_{l/B}g(x)=\ZM_{l/B}g_l(x).
\end{equation}
According to \trm{T4} for $f_l(x)= g_l(\tau_l(x))$ we have
\begin{equation}\label{gf2}
|\{x\in \ZT:\, \zR_B^lf_l(x)>\lambda\}|=\{x\in \ZT^\infty:\,
\ZM_{l/B} g_l(x)>\lambda \}|.
\end{equation}
From \e{DDl} we have
\begin{equation*}
\zR_Bf_l(x)=\zR_B^lf_l(x).
\end{equation*}
So, using also \e{gf1}, \e{gf2} and relation $B\subset D$, we get
\begin{multline*}
\{x\in \ZT^\infty:\, \ZM_{l/B}g(x)>\lambda \}| =|\{x\in \ZT:\,
\zR_B^lf_l(x)>\lambda\}|\\
=|\{x\in \ZT:\,
\zR_Bf_l(x)>\lambda\}|\le |\{x\in \ZT:\, \zR_Df_l(x)>\lambda\}|
\end{multline*}
and therefore
\begin{equation*}
\sup_{B\subset D,\, l\in \ZN,\,\|g\|_\Phi\le 1}\bigg|\bigg\{x\in
\ZT^\infty:\, \ZM_{l/D}g(x)>\lambda \bigg\}\bigg|\le
\sup_{\|f\|_\Phi\le 1}\bigg|\bigg\{x\in \ZT:\, \zR_Df(x)>\lambda
\bigg\}\bigg|.
\end{equation*}
\end{proof}
\end{section}

\begin{section}{A covering lemma}

The covering lemma we establish in this section is needed to prove
\trm{T2}. We consider the function
\begin{equation}\label{alp}
\alpha (x)={} \left\{
\begin{array}{lcl}
x^{x-1},&\hbox { if } & x>1,\\
x, &\hbox { if } & 0\le x\le 1.
\end{array}
\right.
\end{equation}
This is an increasing continuous function from $\ZR^+$ to $\ZR^+$.
It is easy to observe its inverse satisfies the condition
\begin{equation}\label{alp1}
\lim_{x\to\infty}\frac{\alpha^{-1}(x)\ln x }{\ln\ln x }=1.
\end{equation}
Define the functions
\begin{equation}\label{PhiPsi}
\Psi(x)=\int_0^{|x|}\alpha(t)dt,\quad\Phi(x)=\int_0^{|x|}\alpha^{-1}(t)dt,\quad
x\in \ZR.
\end{equation}
These are complementary $N$-functions (see definition in
\cite{KrRu}, chap. 1, par. 2 ). Performing simple estimations we
get
\begin{equation}\label{Phi}
\frac{x\ln (x/2)}{2\ln\ln (x/2)}<\Phi(x)< \frac{x\ln x}{\ln\ln
x},\quad x>\gamma,
\end{equation}
where $\gamma $ is an absolute constant. According to the Young's
inequality (\cite{KrRu}, (2.6)) we have
\begin{equation}\label{Hold}
uv\le\Phi(u)+\Psi(v),\quad u>0,v>0.
\end{equation}
Everywhere below we will use notation $a\lesssim b$ for the
inequality $a\le c\cdot b$ with an absolute constant $c>0$. The
following lemma is a variant of the lemma 3 from \cite{Kar}.
\begin{lemma}\label{L2}
If $A_1,A_2,\ldots A_n$ and $A$ are independent sets in some
probability space and $ \sum_{k=1}^n|A_k|\le 1/2$ then
\begin{equation}\label{intAPsi}
\int_E\Psi\left(\frac{1}{3}\cdot\left(1+\sum_{k=1}^n\ZI_{A_k}(x)\right)\right)\lesssim
|E|,
\end{equation}
where
\begin{equation}
E=A\bigcap\left(\bigcup_{k=1}^nA_k\right).
\end{equation}
\end{lemma}
\begin{proof} To prove \e{intAPsi} it is enough to get
\begin{equation*}
m(\lambda)=\left|\left\{x\in
A:\,1+\sum_{k=1}^n\ZI_{A_k}(x)>\lambda \right\}\right|\lesssim
|E|\cdot\bigg(\frac{2}{\lambda-1}\bigg)^\frac{\lambda-1}{2} ,\quad
\lambda >3.
\end{equation*}
Indeed, using the relation $\Psi'(x)=\alpha (x)$, $x>0$, as a
consequence of \e{PhiPsi}, combined with \e{alp},  we obtain
\begin{align*}
\int_E\Psi\left(\frac{1}{3}\left(1+\sum_{k=1}^n\ZI_{A_k}(x)\right)\right)dx=
\frac{1}{3}\int_0^\infty
\Psi'\left(\frac{\lambda}{3}\right)m(\lambda)d\lambda \\=
\frac{1}{3}
\int_0^\infty\alpha\left(\frac{\lambda}{3}\right)m(\lambda)d\lambda
\lesssim |A|\int_0^\infty \alpha\left(\frac{\lambda}{3}\right)
\bigg(\frac{2}{\lambda}\bigg)^\frac{\lambda}{2}\lesssim |A|.
\end{align*}
Putting $\delta_k =|A_k|$, we have $\sum_{k=1}^n\delta_k<1/2 $.
Then using the independence, we get
\begin{gather*}
|E|=|A\cap A_1|+|A\cap (A_2\setminus A_1)|+\ldots +|A\cap
(A_n\setminus \cup_{k=1}^{n-1}A_k)|
\\
=\delta_1|A|+\delta_2|A|(1-|A_1|)+\ldots
+\delta_n|A|(1-|\cup_{k=1}^{n-1}A_k|)
\\
\ge \delta_1|A|+\frac{\delta_2}{2}|A|+\ldots
+\frac{\delta_n}{2}|A|\ge
\frac{1}{2}|A|(\delta_1+\delta_2+\ldots+\delta_n).
\end{gather*}
We assume $\lambda >3$. Hence
\begin{align*}
m(\lambda )=\\
{}={}&\sum_{k=[\lambda ] }^n\sum_{i_1<\cdots<i_k}\left|A\cap
A_{i_1}\cap \ldots \cap A_{i_k}\cap\left(\bigcap_{j\not\in \{
i_1,\ldots,i_k\}
}(A_j)^c\right)\right|\\
{}={}&\sum_{k=[\lambda ]
}^n\sum_{i_1<\cdots<i_k}|A|\cdot|A_{i_1}|\cdot \ldots \cdot
|A_{i_k}|\prod_{j\not\in \{ i_1,\ldots,i_k\}
}(1-|A_j|)\\
{}={}&|A|\sum_{k=[\lambda
]}^n\sum_{i_1<\cdots<i_k}\delta_{i_1}\cdots
\delta_{i_k}\prod_{j\not\in \{ i_1,\ldots,i_k\} }(1-\delta_j)\\
    {}\le {}&|A|\sum_{k=[\lambda]}^n
\sum_{i_1<\cdots<i_k}\delta_{i_1}\cdots \delta_{i_k} \le
|A|\sum_{k=[\lambda]}^\infty \frac{(\delta_1+\cdots
+\delta_n)^k}{k!}\\
{}<{}&|A|(\delta_1+\cdots +\delta_n)\sum_{k=[\lambda]}^\infty
\frac{1}{k!}\le 2|E|\sum_{k=[\lambda]}^\infty
\frac{1}{\big[\frac{k}{2}\big]!\big(\big[\frac{k}{2}\big]+1\big)\cdots
k}\\
{}\le {}&
2|E|\bigg(\frac{2}{\lambda-1}\bigg)^\frac{\lambda-1}{2}\sum_{k=[\lambda]}^\infty
\frac{1}{\big[\frac{k}{2}\big]!}\lesssim
|E|\bigg(\frac{2}{\lambda-1}\bigg)^\frac{\lambda-1}{2}.
\end{align*}
The proof is complete.
\end{proof}

For a set of indexes $S\subset \ZN $ we denote by $\zR(S)$ the
algebra generated by the rectangles \e{Rl} with $l_i=0,\, i\not\in
S$. For any set $R\subset \zR$ we define its spectrum $\sp (R)$ to
be the smallest set of indexes $S$ for which $R\subset \zR(S)$.
That is
\begin{equation}\label{sp}
\sp (R)=\bigcap_{S:\,R\in \zR(S)}S.
\end{equation}
It is easy to observe
\begin{align}
    &\hbox { if }\sp(B_1),\ldots,\sp(B_k)\hbox { are mutually
    disjoint,
then }\,B_1,\ldots,B_k\hbox { are independent, }\label{sp1}\\
    &\hbox { if } \, \sp(R)\subseteq \sp(Q)\hbox { and } \,Q\not\subseteq R
    \hbox { then }R\cap Q=\varnothing.\label{sp2}
\end{align}
We denote
\begin{align}
&l=l_d=p_1p_2\ldots p_d, \label{lE}\\
&E_d=\{m\in\ZN:\,m=p_\nu p_\mu p_{\mu+1}\ldots p_d,\,
1\le\nu<\mu\}.\label{Ed}
\end{align}
Let $\ZF_d$ be the family of all rectangles from $\ZB_l$ defined
\begin{equation}\label{ZFd}
\ZF_d=\{B_m(j_1,\ldots ,j_d):\,m\in E_d\}
\end{equation}
 According to \e{ZFd} any $B\in\ZF_d$ has the form
\begin{equation}\label{setB}
B=\left\{x\in\ZT^\infty: \, \frac{j_k}{p_k}\le
x_k<\frac{j_k+1}{p_k},\, k\in\{\nu\}\cup\{\mu, \mu+1,\ldots
,d\}\right\},
\end{equation}
where $0\le j_k<p_k$, $k=1,2,\ldots $. In the case $\mu =d+1$ we
understand $\{\mu, \mu+1,\ldots ,d\}=\varnothing $. As $\mu$ and
$\nu$ in \e{setB} are uniquely determined for a given $B\in
\ZF_d$, sometimes  we will use $\mu(B)$, $\nu(B)$ for them. We
define the base $\bs (B)$ and the tail $\tl (B)$ of $B$ by
\begin{equation}\label{base} \bs (B)=
\left\{x\in\ZT^\infty: \, \frac{j_k}{p_k}\le
x_k<\frac{j_k+1}{p_k},\, k\in\{\mu, \mu+1,\ldots ,d\}\right\},
\end{equation}
and
\begin{equation}\label{tail}
 \tl(B)=\left\{x\in\ZT^\infty: \, \frac{j_\nu }{p_\nu }\le
x_\nu<\frac{j_\nu+1}{p_\nu}\right\}.
\end{equation}
Obviously for any $B\in \ZF_d $ we have
\begin{equation}\label{RcapR}
B=\bs(B)\cap\tl(B).
\end{equation}
Observe that if $A,B\in \ZF_d$ then
\begin{align}
&\bs(A)\cap \bs(B)\neq \varnothing & \Rightarrow &\quad
\bs(A)\subseteq \bs(B) \hbox{
or }\bs(B)\subseteq \bs(A), \label{Fpr1}\\
&\bs(A)\subset \bs(B) &\Rightarrow &\quad \mu (A)< \mu
(B),\label{Fpr2}\\
&\bs(A)\subset \bs(B),\quad A\not\subset B &\Rightarrow & \quad
\tl(A)\neq\tl(B).\label{Fpr3}
\end{align}
\begin{lemma}\label{L5} Any collection of rectangles $\Theta=\{A_\alpha\}\subset \ZF_d $ contains a
finite subcollection $\tilde\Theta=\{\tilde A_1,\ldots,\tilde
A_m\}$ with
\begin{align}
&\left|\bigcup_{j=1}^m\tilde A_j\right|\ge \frac{1}{5}\left|\bigcup_\alpha
A_\alpha\right| ,\label{tlBB}\\
    &\int_{\ZT^\infty}\Psi\left(\frac{1}{3}\sum_{j=1}^m\ZI_{\tilde A_j}(x)\right)dx\lesssim
1.\label{tlB}
\end{align}
\end{lemma}
\begin{proof}
Since $\ZF_d$ is finite and $\tilde \Theta \subset \ZF_d$ we can
assume $\Theta=\{A_1,A_2,\ldots ,A_n \}$ and $\mu (A_i)\ge \mu
(A_{i+1})$ for any $i$. The subcollection $\tilde \Theta $ will be
chosen from $\{A_1,A_2,\ldots ,A_n\}$ as follows. We choose
$\tilde A_1=A_1$. If the sets $\tilde A_1=A_{l_1},\ldots,\tilde
A_k=A_{l_{k-1}}$ with $l_1<\ldots <l_{k-1}$ have been chosen then
we select $\tilde A_k$ to be the first set among
$A_{l_{k-1}+1},\ldots ,A_n$ satisfying the conditions
\begin{align}
\tilde A_k\not\subset \tilde A_1\cup\ldots\cup \tilde
A_{k-1},\label{Ak1}\\
\left|\bigcup_{j\le k,\tl(\tilde A_j)\cap\bs(\tilde
A_k)\neq\varnothing,\, \bs(\tilde A_j)\supseteq \bs(\tilde
A_k)}\tl(\tilde A_j)\right|<\frac{3}{4}.\label{Ak2}
\end{align}
This process generates a sequence $\tilde A_1,\tilde A_2,\ldots
,\tilde A_m$. According to \e{Ak2}, for any fixed $k$ we have
\begin{equation}\label{us}
\left|\bigcup_{1\le j\le m,\,\tl(\tilde A_j)\cap\bs(\tilde
A_k)\neq\varnothing,\,\bs(\tilde A_j)\supseteq \bs(\tilde
A_k)}\tl(\tilde A_j)\right|<\frac{3}{4}.
\end{equation}
We consider a base $U=\bs(\tilde A_k)$ satisfying the inequality
\begin{equation}\label{e14}
\left|\bigcup_{\tl(\tilde A_j)\cap U\neq\varnothing,\,\bs(\tilde
A_j)\supseteq U}\tl(\tilde A_j)\right|\ge \frac{1}{4}.
\end{equation}
It is easy to observe that from
\begin{equation*}
\tl(\tilde A_j)\cap U\neq\varnothing,\quad\bs(\tilde A_j)\supseteq
U,
\end{equation*}
it follows that $\nu(\tilde A_j)<\mu(\tilde A_k)$. Therefore the
sets
\begin{equation*}
U, \quad \bigcup_{\bs(\tilde A_j)\supseteq U }\tl(\tilde A_j)
\end{equation*}
have disjoint spectrums and so they are independent according to
\e{sp1}. Thus, using \e{e14} we conclude
\begin{multline}\label{14U}
\left|U\bigcap\left(\bigcup_{j=1}^m\tilde A_j\right)\right|\ge
\left|U\bigcap\left(\bigcup_{\bs(\tilde
A_j)\supseteq U }\tilde A_j\right)\right|\\
= \left|U\bigcap\left(\bigcup_{\bs(\tilde A_j)\supseteq U
}\tl(\tilde A_j)\right)\right|= |U|\cdot\left|\bigcup_{\bs(\tilde
A_j)\supseteq U }\tl(\tilde A_j)\right|\ge \frac{1}{4}|U|.
\end{multline}
We denote by $U_1,U_2,\ldots ,U_\gamma $ the family of all
maximal bases $U=\bs(\tilde A_k)$ satisfying \e{e14}. It is
clear they are mutually disjoint their union contains all $U$ satisfying \e{e14}. Thus, using \e{14U} we get
\begin{equation}\label{UAt}
\left|\bigcup_{i=1}^\gamma U_i\right|\le
4\left|\bigcup_{j=1}^m\tilde A_j\right|.
\end{equation}
Now suppose $A_t$ is an arbitrary set which is not in the
subcollection $\{\tilde A_k\}$. We have $l_{k-1}< t<l_k$ for some
$k$. According to the process of the selection we have either
\begin{equation}\label{Atin}
A_t\subset \bigcup_{i=1}^{k-1}\tilde A_i
\end{equation}
or
\begin{equation*}
\left|\tl(\tilde A_t)\bigcup\left(\bigcup_{j< k,\tl(\tilde
A_j)\cap\bs(\tilde A_t)\neq\varnothing,\, \bs(\tilde A_j)\supseteq
\bs(\tilde A_t)}\tl(\tilde A_j)\right)\right|\ge\frac{3}{4}.
\end{equation*}
Since $\tl(\tilde A_t)\le \frac{1}{2}$ we obtain
\begin{equation*}
\left|\bigcup_{j< k,\tl(\tilde A_j)\cap\bs(\tilde
A_t)\neq\varnothing,\, \bs(\tilde A_j)\supseteq \bs(\tilde
A_t)}\tl(\tilde A_j)\right|\ge\frac{1}{4},
\end{equation*}
which means $\bs(A_t)\subseteq U=\bs(\tilde A_{k-1})$ where $U$
satisfies \e{e14}. Hence we have either \e{Atin} or
\begin{equation*}
A_t\subset\cup_{i=1}^\gamma U_i,
\end{equation*}
and therefore, applying \e{UAt}, we get
\begin{equation*}
\left|\bigcup_tA_t\right|\le \left|\bigcup_{j=1}^m\tilde
A_j\right|+\left|\bigcup_{i=1}^\gamma U_i\right|\le
5\left|\bigcup_{j=1}^m\tilde A_j\right|.
\end{equation*}
which gives \e{tlBB}. To prove \e{tlB} denote
\begin{equation}\label{tilA}
B_k=\bs(\tilde A_k)\setminus \left(\bigcup_{\bs(\tilde A_i)\subset
\bs(\tilde A_k)} \bs(\tilde A_i)\right),\quad k=1,2,\ldots ,m.
\end{equation}
It is clear $B_1,B_2,\cdots , B_m$ are pairwise disjoint. We note
some of this sets can be empty. Using \e{RcapR} we have
\begin{equation*}
\bigcup_{k=1}^m B_k =\bigcup_{k=1}^m \bs(\tilde A_k)
\supset\bigcup_{k=1}^m \tilde A_k.
\end{equation*}
Thus, to obtain \e{tlB}, it is enough to prove
\begin{equation}\label{I1}
I_k=\int_{B_k}\Psi\left(\frac{1}{3}\sum_{j=1}^m\ZI_{\tilde
A_j}(x)\right)dx\lesssim |B_k|.
\end{equation}
Observe that
\begin{equation}\label{I2}
I_k=\int_{B_k}\Psi\left(\frac{1}{3}\sum_{j:\,\bs(\tilde A_j)\supseteq
\bs(\tilde A_k)}\ZI_{\tilde A_j}(x)\right)dx.
\end{equation}
Indeed, according to \e{Fpr1}, any $\tilde A_j$ satisfies one of
the relations
\begin{align}
&\bs(\tilde A_j)\cap \bs(\tilde A_k)=\varnothing ,\label{C1}\\
&\bs(\tilde A_j)\subset \bs(\tilde A_k),\label{C2}\\
&\bs(\tilde A_j)\supseteq \bs(\tilde A_k).\label{C3}
\end{align}
 In the case \e{C1} or \e{C2}, using \e{tilA}, we have $\tilde A_j\cap
B_k=\varnothing$. So the integral \e{I1} depends only on the sets
$\tilde A_j$ with \e{C3}, which implies \e{I2}. If $\bs(\tilde
A_j)\supseteq\bs(\tilde A_k)$ then by \e{tilA} $\bs(\tilde
A_j)\supseteq B_k$. Thus, such that $\tilde A_j=\bs(\tilde
A_j)\cap\tl(\tilde A_j)$ (see \e{RcapR}) from \e{I2} we get
\begin{equation*}
I_k=\int_{B_k}\Psi\left(\frac{1}{3}\sum_{\bs(\tilde A_j)\supseteq
\bs(\tilde A_k)}\ZI_{\tl(\tilde A_j)}(x)\right)dx.
\end{equation*}
Now denote
\begin{align}\label{Cnu}
C_\nu=\bigcup_{j:\,\nu(\tilde A_j)=\nu,\,\tl(\tilde
A_j)\cap\bs(\tilde A_k)\neq\varnothing,\, \bs(\tilde A_j)\supseteq
\bs(\tilde A_k) }\tl(\tilde A_j)
\end{align}
and consider all nonempty sets
$C_{\nu_1},C_{\nu_2},\ldots,C_{\nu_p}$, with decreasing numbering
$\nu_1>\nu_2>\ldots>\nu_p$. From \e{us} it follows that
\begin{equation}\label{Cnu34}
\left|\bigcup_{i=1}^p C_{\nu_i}\right|<\frac{3}{4}.
\end{equation}
Observe that if the sets $\tilde A_j$ and $\tilde A_i$ satisfy the
relations
\begin{equation}\label{numu}
\bs(\tilde A_j)\supseteq \bs(\tilde A_i)\hbox { and }\nu(\tilde
A_j)\ge \mu(\tilde A_i)
\end{equation}
then
\begin{equation}\label{AjAi}
\tl(\tilde A_j)\cap  \bs(\tilde A_i)=\varnothing .
\end{equation}
Indeed, from \e{numu} and the definition of the set $\ZF_d$ in
\e{ZFd} it follows that
\begin{equation*}
\sp(\tilde A_j)\subseteq \{\nu(\tilde A_j),\nu(\tilde
A_j)+1,\ldots d\}\subseteq \{\mu(\tilde A_i),\mu(\tilde
A_i)+1,\ldots d\}=\sp(\bs(\tilde A_i)).
\end{equation*}
Thus, using \e{sp2} we will have either $\tilde
A_j\supseteq\bs(\tilde A_i)\supset \tilde A_i$ or $\tilde
A_j\cap\bs(\tilde A_i)=\varnothing $. The first inclusion is not
possible because of \e{Ak1}. So we have $\tilde A_j\cap\bs(\tilde
A_i)=\varnothing $. Therefore, since $\tilde A_j= \bs(\tilde
A_j)\cap\tl(\tilde A_j)$ and $\bs(\tilde A_j)\supseteq \bs(\tilde
A_i)$ (see \e{numu}) we get \e{AjAi}. Combining \e{AjAi} with
\e{Cnu} we get
 \begin{equation*}
C_{\nu_j}\cap  \bs(\tilde A_i)=\varnothing,
\end{equation*}
provided
\begin{equation*}
 \bs(\tilde A_k)\supseteq\bs(\tilde A_i),\quad
\mu(\tilde A_i)\le\nu_j.
 \end{equation*}
Therefore by \e{tilA}
\begin{multline*}
B_k\cap (C_{\nu_j}\setminus
\cup_{s=1}^{j-1}C_{\nu_s})\\
=\left(\bs(\tilde A_k)\setminus
\bigcup_{\bs(\tilde A_i)\subset \bs(\tilde A_k),\,\mu(\tilde A_i)>
\nu_j } \bs(\tilde A_i)\right)\cap (C_{\nu_j}\setminus
\cup_{s=1}^{j-1}C_{\nu_s}).
\end{multline*}
Since $\sp(C_{\nu_s})=\nu_s$, $\nu_p<\nu_{p-1}<\ldots <\nu_1$ and
$\sp(\bs(\tilde A_i))= \{\mu(\tilde A_i),\mu(\tilde
A_i)+1,\ldots,d\}$ (see \e{base}), each set on the right has
spectrum in $\{\nu_j,\nu_j+1,\ldots,d\}$. So we have
\begin{equation*}
\sp \big(B_k\cap (C_{\nu_j}\setminus
\cup_{s=1}^{j-1}C_{\nu_s})\big)\subset\{\nu_j,\nu_j+1,\ldots,d\}.
\end{equation*}
Hence the sets
\begin{equation*}
B_k\cap (C_{\nu_i}\setminus
\cup_{s=1}^{i-1}C_{\nu_s}),C_{\nu_{i+1}},\ldots , C_{\nu_p}
\end{equation*}
have mutually disjoint spectrums, so they are independent by
\e{sp1}. According to \e{Cnu34} these sets satisfy the hypothesis
of \lem{L2}. Hence, applying \e{intAPsi}, we get
\begin{equation*}
\int_{B_k\cap(C_{\nu_i}\setminus
\cup_{s=1}^{i-1}C_{\nu_s})}\Psi\left(\frac{1}{3}\left(1+\sum_{t=i+1}^p\ZI_{C_{\nu_t}}(x)\right)
\right)dx\lesssim |B_k\cap(C_{\nu_i}\setminus
\cup_{s=1}^{i-1}C_{\nu_s})|
\end{equation*}
and therefore
\begin{multline*}
I_k=\int_{B_k}\Psi\left(\frac{1}{3}\sum_{i=1}^p\ZI_{C_{\nu_i}}(x)\right)dx\\
=\sum_{i=1}^p \int_{B_k\cap(C_{\nu_i}\setminus
\cup_{s=1}^{i-1}C_{\nu_s})}\Psi\left(\frac{1}{3}\left(1+\sum_{t=i+1}^p\ZI_{C_{\nu_t}}(x)
\right)\right)dx\\
\lesssim
\sum_{i=1}^p|B_k\cap(C_{\nu_i}\setminus\cup_{s=1}^{i-1}C_{\nu_s})|\le
|B_k|,
\end{multline*}
where $C_{\nu_0}=\varnothing $. Hence the inequality \e{I1} and so
the lemma is proved.
\end{proof}
In the following lemma $E\subset Z$ is the set defined in \e{E1}
and $\ZM_{l/E}f(x)$ is the maximal function from \e{MBl}.
\begin{lemma}\label{L6}
If $\Phi(t)$ is the function from \e{PhiPsi} then
\begin{equation}\label{lem6}
|\{x\in \ZT^\infty:\,\ZM_{l/E} f(x)>\lambda \}|\lesssim
\frac{1}{\lambda}\left(1+\int_{\ZT^\infty}\Phi(f(t))dt\right),\quad
\lambda >0,
\end{equation}
for any $f\in L^\Phi(\ZT^\infty)$ and $l\in\ZN$.
\end{lemma}

\begin{proof}
We suppose $l$ has the factorization \e{lE}. From \e{Ed} and
\e{E1} we get $l/E=E_d$. So taking into account \e{ZFd} we have
\begin{equation*}
 \ZM_{l/E}f(x)=\sup_{F\in \ZF_d:\,F\ni x }\frac{1}{|F|}\int_F
|f(t)|dt,\quad x\in \ZT^\infty,\quad f\in L^1\big(\ZT^\infty\big).
\end{equation*}
Hence, for any $\lambda >0$ there exists a collection $F=\{F_k \}$
from $\ZF_d $ such that
\begin{align*}
&\{x\in \ZT^\infty:\,\ZM_{l/E} f(x)>\lambda \}=\bigcup_k F_k, \\
&\frac{1}{|F_k |}\int_{F_k }f(t)dt>\lambda .
\end{align*}
According to \lem{L5} we can choose a subfamily $\{\tilde F_k\}$
such that
\begin{gather}
\left|\bigcup_k \tilde F_k\right|\ge \frac{1}{5}\left|\bigcup_k
F_k\right|,\label{tilAk1}\\
\int_{\ZT^\infty}\Psi\left(\sum_k\ZI_{\tilde
F_k}(x)\right)dx\lesssim 1.\label{tilAk2}
\end{gather}
Thus, applying \e{tilAk1},\e{tilAk2} and \e{Hold} we obtain
\begin{multline*}
|\{x\in \ZT^\infty:\,\ZM_{l/E} f(x)>\lambda \}|\label{MQest}\\
= \left|\bigcup_k F_k\right|\le 5\sum_k|\tilde F_k | \le
\sum_k\frac{5}{\lambda }\int_{\tilde F_k}f(t)dt =\frac{5}{\lambda
}\int_{\ZT^\infty} f(t)\sum_k\ZI_{\tilde F_k}(t)dt\\\le
\frac{5}{\lambda } \left(\int_{\ZT^\infty}\Phi(f(x))dx+
\int_{\ZT^\infty}\Psi\bigg(\sum_k\ZI_{\tilde
F_k}(x)\bigg)dx\right) \lesssim
\frac{5}{\lambda}\left(1+\int_{\ZT^\infty}\Phi (f(t))dt\right).
\end{multline*}

\end{proof}
\end{section}

\begin{section}{Proofs of Theorems}
To avoid of the repetition of the same standard argument in the
proofs of the theorems we will use E. M. Stein's well-known weak
type maximal functions principle (see. \cite{Ste1} or \cite{Ste2}
chap. X, par. 3.6). Consider a sequence of convolution operators
\begin{equation*}
T_j=f\ast \mu_j:L^1(\ZT)\to \{\hbox {measurable functions on \ZR}
\}
\end{equation*}
where $\mu_j$ are positive finite measures on $\ZT$.
\begin{lemma}[E. M. Stein]
Let $\Phi:\ZR^+\to\ZR^+$ to be an increasing convex function such
that $\Phi(\sqrt{x})$ is concave. Then if for every $f\in \Phi(L)$
\begin{equation*}
Mf(x)=\sup_j|T_jf(x)|<\infty
\end{equation*}
on a set of positive measure then
\begin{equation*}
|\{x\in \ZR: Mf(x)>\lambda \}|\le
\int_\ZR\Phi\left(\frac{c|f|}{\lambda}\right),\quad \lambda>0,
\end{equation*}
where $c>0$ is a constant.
\end{lemma}
\begin{proof}[Proof of Theorem~\ref{T2}]
We suppose $B\subset E$ is an arbitrary finite set. If $l$ is a
multiple for the members of $B$ then $l/B\subset l/E$, and so by
\e{MBl} we obtain
\begin{equation*}
\ZM_{l/B}f(x)\le \ZM_{l/E}f(x).
\end{equation*}
 Hence, according to \e{lem6} we have
\begin{equation*}
|\{x\in \ZT:\, \ZM_{l/B}f(x)>\lambda\}|<
\frac{c}{\lambda}\left(1+\int_{\ZT}\Phi(f(t))dt\right),
\end{equation*}
for any finite $B\subset E$ and $f\in L^\Phi$. Combining this with
the corollary after \trm{T4} we obtain
\begin{equation}\label{MEd} |\{x\in \ZT:\zR_Ef(x)>\lambda
\}|\le
\frac{c}{\lambda}\left(1+\int_{\ZT^\infty}\Phi(f(t))dt\right),
\end{equation}
where $c>0$ is an absolute constant. We have each $B_mf(x)$ is a
convolution operator with the kernel
\begin{equation*}
\mu_m=\frac{1}{m}\sum_{i=1}^m\delta_{i/m}
\end{equation*}
where $\delta_a$ is the unit measure (Dirac function) concentrated
at $a$. It is easy to check as well $\Phi $ satisfies the
hypothesis of Stein's lemma. Therefore applying the Stein's
principle from \e{MEd} we get \e{T2form}.
  The proof is thus complete.
\end{proof}
Suppose $f(x)\in L^1(\ZT)$, $D$ is a finite set of naturals and
$l$ is a common multiple for the members of $D$. Consider the
conditional expectation $E^{\zI_l}f(x)$ of the function $f(x)$
with respect the algebra $\zI_l$ defined. For any convex function
$\phi:\ZR^+\to\ZR^+$ we have
\begin{equation}\label{cep}
\|E^{\zI_l}f(x)\|_\phi\le \|f\|_\phi.
\end{equation}
 To deduce \textit{everywhere} divergence in \trm{T1} and
\trm{T3} we use the following general lemma.
\begin{lemma}\label{L8}
Let $D$ be a set of indexes and $\phi :\ZR^+\to\ZR^+$ is a convex
increasing function. If there exists a function $f\in L^\phi $
such that $\zR_Df(x)=\infty$ on a set of positive measure, then it
can be found a function $\tilde f\in L^\phi $ with $\zR_D \tilde
f(x)=\infty $ everywhere.
\end{lemma}
\begin{proof} Suppose for some $f\ge 0$ we have
\begin{equation*}
\zR_Df(x)=\infty,\quad x\in E,
\end{equation*}
and $|E|>0$. According to Borel-Cantelli lemma (see. \cite{Ste2},
p. 442 or \cite{Zyg2}, section XIII, 1.24) there exists a sequence
$x_k\in \ZT$ such that $ \sum_k\ZI_{E+x_k}(x)=\infty $ a.e..
Denoting $\tilde f(x)=\sum_k 2^{-k}f(x+x_k)$, we get $\tilde f\in
L^\phi$ and
\begin{equation*}
\zR_D\tilde f(x)=\infty\hbox { a.e. }.
\end{equation*}
Hence by \e{RDMD} there exist a sequence of finite sets
$D_1\subset D_2\subset\ldots $ with $\cup_n D_n=D$ and a integers
$l_n$ divided by the members of $D_n$ such that
\begin{equation*}
|\{x\in \ZT:\,\zR_{D_n}^{l_n}\tilde
f(x)>n^3\}|>1-\frac{1}{\phi(n^3)}.
\end{equation*}
Since $\zR_{D_n}^{l_n}\tilde f(x)$ is $\zI_{l_n}$-measurable, so
the set
\begin{equation*}
A_n=\{x\in \ZT:\,\zR_{D_n}^{l_n}\tilde f(x)>n^3\}
\end{equation*}
is. Hence we get
\begin{align*}
&|A_n^c|\le1/\phi(n^3),\\
&\zR_{D_n}^{l_n}\ZI_{A_n^c}(x)=1,\quad x\in A_n^c.
\end{align*}
Thus, denoting
\begin{equation*}
f_n(x)= \tilde f(x)+n^3\cdot\ZI_{A_n^c}(x),
\end{equation*}
we have
\begin{align*}
&\|f_n\|_\phi\le\|f_n\|_\phi+\|n^3\cdot\ZI_{A_n^c}\|_\phi=\|f_n\|_\phi+1=\|\tilde
f\|_\phi+1,\\
& \zR_{D_n}^{l_n}f_n(x)>n^3,\hbox { for all }
x\in \ZT^\infty.
\end{align*}
Now denote
\begin{equation*}
g(x)=\sum_{n=1}^\infty\frac{1}{n^2}\cdot E^{\zI_{l_n}} f_n(x).
\end{equation*}
According to \e{cep} we have
\begin{equation*}
\|g\|_\phi\le \sum_{n=1}^\infty\frac{1}{n^2}\cdot \|E^{\zI_{l_n}}
f_n\|_\phi\le\sum_{n=1}^\infty\frac{1}{n^2}\cdot \|
f_n\|_\phi<\infty,
\end{equation*}
and using \e{DDl} we get
\begin{equation*}
\zR_Dg(x)\ge \zR_{D_n} g(x)\ge\frac{1}{n^2}\zR_{D_n}E^{\zI_{l_n}}
f_n(x) =\frac{1}{n^2}\zR_{D_n}^{l_n} f_n(x)>n,\quad x\in
\ZT^\infty,
\end{equation*}
for any $n\in \ZN$, i.e. $\zR_D g(x)=\infty $ everywhere on $\ZT$.
The proof is complete.
\end{proof}
\begin{proof}[Proof of Theorem~\ref{T3}]
We consider the rectangles
\begin{equation*}
B_i^k=\left\{x\in \ZT^\infty :\, \frac{i}{p_k}\le x_k<
\frac{i+1}{p_k}\right\},\quad i=0,1,\ldots ,p_k-1.
\end{equation*}
Since $\sp(B_i^k)=\{k\}$ we have $B_i^k\in \ZF_{2d}$ if $1\le k\le
2d$. Denote
\begin{align}
&G_k=\bigcup_{0\le i<\left[\frac{p_k}{d}\right]}B_i^k,\quad k=1,2,\ldots ,2d,\label{Gk}\\
&G=\bigcup_{k=d+1}^{2d}G_k,\quad C=\bigcap_{k=d+1}^{2d}G_k.\label{G}\\
\end{align}
It is clear $p_{d+1}>2d$. Since the number of $B_i^k$ in the union
\e{Gk} is $\left[\frac{p_k}{d}\right]$ and $|B_i^k|=1/p_k$ we
conclude
\begin{equation}\label{1d2d}
\frac{1}{d}\ge
|G_k|=\left[\frac{p_k}{d}\right]\frac{1}{p_k}>\frac{1}{d}\left(1-
\frac{d}{p_k}\right)>\frac{1}{2d},\hbox { if } k>d.
\end{equation}
Because of independence of the sets $G_k$ we get
\begin{align}
&|G|=|\bigcap
_{k=d+1}^{2d}G_k|=1-\prod_{k=d+1}^{2d}(1-|G_k|)>1-(1-(2d)^{-1})^d>
1-\frac{1}{\sqrt{e}}>\frac{1}{3},\label{|G|}\\
&|C|=\prod_{k=d+1}^{2d}|G_k|\le d^{-d}.\label{|C|}
\end{align}
Choose an arbitrary $x\in G$. We have $x\in G_k$ for some $k$ and
therefore $x\in B_i^k$ for some $0\le
i<\left[\frac{p_k}{d}\right]$ and $d<k\le 2d$. On the other hand,
using \e{1d2d} and the independence of the sets $G_j$, $d<j\le
2d$, $j\neq k$, with $B_i^k$, we obtain
\begin{equation*}
|C\bigcap B_i^k|= \left|\left(\bigcap_{d<j\le 2d,\,j\neq
k}G_k\right)\bigcap B_i^k\right|=|B_i^k|\prod_{d<j\le 2d,\,j\neq
k}|G_k|>\frac{|B_i^k|}{(2d)^{d-1}}.
\end{equation*}
From this we get
\begin{equation*}
\frac{1}{|B_i^k|}\int_{B_i^k}\ZI_C(x)dx>(2d)^{1-d}.
\end{equation*}
So we conclude
\begin{equation}\label{Ml2d}
\ZM_{l_{2d}/E}\ZI_C(x)>(2d)^{1-d},\quad x\in G,
\end{equation}
where $l_{2d}$ is defined in \e{lE}. Taking into account \e{phis}
and \e{|C|}, we have
\begin{equation*}
\int_{\ZT^\infty}\phi((2d)^{d-1}\ZI_C(x))dx=\phi((2d)^{d-1})|C|<d^{-d}\phi((2d)^{d-1})
\stackrel{d\to\infty}{\rightarrow} 0,
\end{equation*}
Thus, we may find a sequence $c_d\to \infty $ such that the
function
\begin{equation}\label{gdC}
g_d(x)=c_d(2d)^{d-1}\ZI_C(x).
\end{equation}
satisfies
\begin{equation*}
\int_{\ZT^\infty}\phi(g_d(x))dx\le 1.
\end{equation*}
From \e{gdC} we get
\begin{equation*}
\ZM_{l_{2d}/E}g_d(x)=c_d(2d)^{d-1}\ZM_{l_{2d}/E}\ZI_C(x)
\end{equation*}
and so, using \e{|G|} and \e{Ml2d}, we obtain
\begin{equation*}
|\{x\in \ZT^\infty:\,\ZM_{l_{2d}/E}g_d(x)>c_d\}|=|\{x\in
\ZT^\infty:\,\ZM_{l_{2d}/E}\ZI_C(x)>(2d)^{1-d}\}|\ge
|G|>\frac{1}{3}.
\end{equation*}
Applying \e{cor} we may find sequence of functions $f_d$ on $\ZT$
with
\begin{equation*}
\|f_d\|_\Phi=\|g_d\|_\Phi\le \int_{\ZT^\infty}\phi(g_d(x))dx\le 1
\end{equation*}
such that
\begin{equation*}
|\{x\in \ZT:\,\zR_Ef_d(x)>c_d\}|>\frac{1}{3}.
\end{equation*}
Hence, according to Stein's principle there exists a function
$f\in L^\Phi(\ZT)$ such $\zR_Ef(x)=\infty$ a.e.. To get everywhere
divergence it remains to use \lem{L8}. \trm{T3} is proved.
\end{proof}
The proof of \trm{T1} is based on some results in the Theory of
Differentiation of Integrals in $\ZR^n$. According to well known
Jessen-Marcinkiewicz-Zygmund theorem (see \cite{JMZ} or \cite{Guz}
chapter 2)
\begin{equation}\label{JMZ}
\lim_{\diam{R}\to 0,\,x\in R}\frac{1}{|R|}\int_Rf(t)dt=f(x),\hbox
{ a.e }
\end{equation}
for any $f\in L\log^{n-1}L(\ZR^n)$, where $R$ are rectangles with
sides parallel to the axis. On the other hand S.~Saks in
\cite{Saks} has proved that in this theorem the Orlicz class
$L\log^{d-1}L$ is the optimal. Certainly the relation \e{JMZ} is
true also if we consider the rectangles \e{Rl} with fixed $d$
instead of all rectangles in $\ZR^n$. As for the divergence
theorem the proof is not immediate. However there is a
generalization of Saks theorem due A.~Stokolos \cite{Sto}(see also
\cite{HarSto}). According to this theorem if $\phi$ satisfies
\e{phis} then there exists a function $f\in L^\phi(\ZR^n)$ such
that
\begin{equation}\label{Sto}
\lim_{\diam{R}\to 0,\,x\in R}\frac{1}{|R|}\int_R f(t)dt=\infty ,
\end{equation}
for any $x\in \ZR^n$, where $R$ are the rectangles of the form
\e{Rl} with fixed $d$. Moreover, it can be taken any integers
greater than or equal $2$ instead of primes $p_1,p_2,\ldots ,p_d$.
We note that all this theorems can be stated also on $\ZT^\infty
$.
\begin{proof}[Proof of Theorem \ref{T1}]
Suppose $D$ is the set of all integers of the form
\begin{equation*}
p_1^{m_1}p_2^{m_2}\ldots p_d^{m_d},\quad m_k \in\ZZ^+,\,
k=1,2,\ldots d.
\end{equation*}
Consider a sequence of subsets $D_n\subset D$ defined
\begin{equation*}
D_n=\{m=p_1^{m_1}p_2^{m_2}\ldots p_d^{m_d}:\,0\le m_k \le n,\,
k=1,2,\ldots d\},
\end{equation*}
and denote
\begin{equation*}
l_n=p_1^np_2^n\ldots p_d^n.
\end{equation*}
We have $\cup_nD_n=D$ and $l_n/D_n=D_n$. Therefore if the function
$f\in L^\phi (\ZT^\infty)$ satisfies the condition \e{Sto} then
\begin{equation*}
\lim_{k\to\infty}\ZM_{l_n/D_n}f(x)=\infty,\hbox { a.e on }
\ZT^\infty.
\end{equation*}
Applying \e{cor}, we get $\zR_Dg_n(x)\to\infty $ a.e. for a
sequence of functions $g_n$ with $\|g_n\|_\Phi\le 1$. Using
Stein's principle, we will get a function $g$ with
$\zR_Dg(x)=\infty $ a.e., and the existence of a function with
\textit{everywhere} divergence Riemann sums follows from \lem{L8}.
\end{proof}
\end{section}
\begin{section}{On Rudin's theorem and sweeping out properties}
In this section we establish equivalency between strong sweeping
out and $\delta$-sweeping out properties of operator sequences,
which seems to be interesting in view of the papers
\cite{Akc},\cite{Akc2}. Then we will deduce Rudin's theorem in
general settings from \trm{T4}.

Let $(X,m)$ be a probability space. We consider linear operators
\begin{equation}\label{oper}
T:L^1(X,m)\to \{\hbox {measurable functions on X} \}.
\end{equation}
\begin{definition}
 A sequence of linear operators $T_n $ is said to be strong
sweeping out if given $\varepsilon >0$ there is a set $E$ with
$mE<\varepsilon $ such that $\limsup_{n\to\infty} T_n\ZI_E(x)=1$
a.e. and $\liminf_{n\to\infty} T_n\ZI_E(x)=0$ a.e..
\end{definition}
\begin{definition}
Let $0<\delta\le 1$. A sequence of linear operators $T_n $ is said
to be $\delta $-sweeping out if given $\varepsilon >0$ there is a
set $E\subset X$ with $mE<\varepsilon $ such that
$\limsup_{n\to\infty} T_n\ZI_E(x)\ge\delta $ a.e..
\end{definition}
\begin{definition}
Let $0<\delta\le 1$. A sequence of linear operators $T_n $ is said
to be weak $\delta $-sweeping out if given $r >0$ there is a set
$E$ such that
\begin{equation*}
m\{x\in X:\, \sup_{n\in \ZN} T_n\ZI_E(x)\ge\delta \}>r\cdot mE.
\end{equation*}
\end{definition}
It turns out that these definitions are equivalent for the
sequences of linear operators having
 the following settings
\begin{enumerate}
    \item if $f\ge 0$ then $Tf\ge 0$,
    \item $T(\ZI_X)=1$,
    \item for any $\varepsilon >0$ there exists $\delta >0$
    such that if $E\subset X$ and
$m(E)<\delta$ then
\begin{equation*}
m\{x\in X; T\ZI_E(x)>\varepsilon \}<\varepsilon .
\end{equation*}
\end{enumerate}
\begin{theorem}\label{T6}
If the sequence of linear operators $T_n$ satisfying (1)-(3) is
$\delta $-sweeping out for any $0< \delta < 1$ then it is strong
sweeping out.
\end{theorem}
\begin{proof} Assume $\{T_n\}$ satisfies the hypothesis of the theorem.
Using a standard argument, one can easily choose a sequence of
integers $1=n_0<n_1<n_2<\ldots $ and measurable sets $E_k\subset
X$ such that
\begin{align}
& mE_k<\varepsilon 2^{-k},\label{Ek}\\
& m\{x\in X; \sup_{n_{k-1}\le m\le
n_k}T_m\ZI_{E_k}(x)>1-2^{-k} \}>1-2^{-k},\label{Bk}\\
& m\{x\in X; \sup_{n_{k-1}\le m\le
n_k}T_m\left(\sum_{j=k+1}^\infty\ZI_{E_j}(x)\right)>2^{-k}
\}<2^{-k}.\label{Ak}
\end{align}
The selection of $n_k$ and $E_k$ is realized in this order:
$E_1,n_1,E_2,n_2,\ldots $. To avoid big expressions we use the
notation $U_k=\sup_{n_{k-1}\le m\le n_k}T_m$. Denote
\begin{align*}
&\tilde E_k=E_k\setminus \cup_{j=k+1}^\infty E_j,\quad E=\cup_{j=0}^\infty \tilde E_{2j+1},\\
&A_k=\left\{x\in X:\, U_k\left(\sum_{j=k+1}^\infty\ZI_{E_j}(x)\right)\le2^{-k}
\right\},\\
&B_k=\left\{x\in X:\, U_k\ZI_{E_k}(x)>1-2^{-k} \right\},\\
&G=(\liminf_{k\to\infty} A_k)\cap (\liminf_{k\to\infty} B_k).
\end{align*}
From \e{Ak} and \e{Bk} we get $mG=1$. Given an arbitrary $x\in G$
we have
\begin{equation*}
x\in A_k\cap B_k,\quad k> k_0,
\end{equation*}
and consequently
\begin{equation}\label{Uk}
U_k\left(\sum_{j=k+1}^\infty\ZI_{E_j}(x)\right)\le2^{-k},\quad
U_k\ZI_{E_k}(x)>1-2^{-k},\quad k> k_0.
\end{equation}
Thus
\begin{multline*}
U_{2k+1}\ZI_E(x)\ge U_{2k+1}\ZI_{\tilde
E_{2k+1}}(x)\\
\ge U_{2k+1}\ZI_{E_{2k+1}}(x)- U_{2k+1}\left(\sum_{j=2k+2}^\infty
\ZI_{E_j}(x)\right)
>1-2^{-(2k+1)}-2^{-(2k+1)}=1-2^{-2k}.
\end{multline*}
This implies
\begin{equation*}
\limsup_{m\to\infty} T_m\ZI_E(x)=1,\quad x\in E.
\end{equation*}
It is easy to observe $E\cap E_{2k}=\varnothing$. So we have
$E\subset \tilde E_{2k}^c$ and from \e{Uk} we derive
\begin{multline*}
U_{2k}\ZI_E(x)\le U_{2k}\ZI_{\tilde E_{2k}^c}(x) =
1-U_{2k}\ZI_{\tilde E_{2k}}(x)
\\
\le 1-U_{2k}\ZI_{E_{2k}}(x)+U_{2k}\left(\sum_{j=2k+1}^\infty
\ZI_{E_j}(x)\right)<1-(1-2^{-{2k}})+2^{-{2k}} =2^{-2k+1}.
\end{multline*}
Hence
\begin{equation*}
\liminf_{m\to\infty} T_m\ZI_E(x)=0,\quad x\in E,
\end{equation*}
and the proof is complete.
\end{proof}
Now suppose $(X,m)$ in \e{oper} coincides with $(\ZT,\lambda )$.
In the next theorem we consider translation invariant operators
$T_n$ defined
\begin{equation*}
T_nf_x(t)=T_nf(x+t),
\end{equation*}
where $f_x(t)=f(x+t)$.
\begin{theorem}\label{T7}
If the sequence of translation invariant operators $\{T_n\}$ with
(1)-(3) is weak $\delta $-sweeping out for any $0< \delta < 1$
then it is strong sweeping out.
\end{theorem}
\begin{proof}
According to the previous theorem it is enough to proof that
$\{T_n\}$ is $\delta $-sweeping out for any $0<\delta <1$. By weak
$\delta $-sweeping property we may choose measurable sets $F_k$
such that
\begin{equation*}
\frac{|\{x\in X:\, \sup_{n>k} T_n\ZI_{F_k}(x)\ge 1-\frac{1}{k}
\}|}{|F_k|}\to\infty .
\end{equation*}
Taking subsequences of $F_k$ (with possible repetitions) allows us
to find a sequence of sets $E_k$, a sequences $\delta_k\nearrow
1$, and $n_k\to\infty$ so that, taking
\begin{equation*}
A_k=\{x\in X:\, \sup_{n>n_k} T_n\ZI_{E_k}(x)\ge \delta_k \},
\end{equation*}
we have
\begin{equation*}
\sum_{k=1}^\infty |A_k|=\infty, \quad \sum_{k=1}^\infty
|E_k|<\varepsilon.
\end{equation*}
 Applying Borel-Cantelli lemma, we can choose a sequence $x_k$ so
that
\begin{equation*}
|\limsup_{k\to\infty}(A_k+x_k)|=1.
\end{equation*}
Since $T_n$ are translation invariant operators, denoting
\begin{equation*}
E=\bigcup_{k=1}^\infty (E_k+x_k)
\end{equation*}
we get
\begin{multline*}
|\{x\in X:\, \limsup_{n\to\infty } T_n\ZI_E(x)=1 \}|\\
 \ge|\limsup_{k\to\infty}\{x\in X:\, \sup_{n>n_k}
T_n\ZI_{E_k+x_k}(x)\ge \delta_k \}|=
|\limsup_{k\to\infty}(A_k+x_k)|=1,
\end{multline*}
and
\begin{equation*}
|E|\le\sum_{j=1}^\infty|E_k+x_k|=\sum_{j=1}^\infty|E_k|<\varepsilon.
\end{equation*}
\end{proof}
Clearly Riemann sums operators satisfy the conditions (1)-(3). (1)
and (2) are clear. Let us verify (3). If for $E\subset \ZT$ we
have $|E|<\delta=\varepsilon^2 $ then
\begin{equation*}
\int_\ZT R_n\ZI_E(x)dx=|E|<\varepsilon^2
\end{equation*}
and therefore, using Chebishev's inequality, we get
\begin{equation*}
|\{x\in\ZT:\, R_n\ZI_E(x)>\varepsilon\}|<\varepsilon,
\end{equation*}
which proves (3). Analyzing Rudin's proof one can easily
understand it allows to get $\delta$-sweeping out property for any
$0<\delta<1$. Thus applying \trm{T6} we conclude that if $\{n_k\}$
satisfies the hypothesis of Rudin's theorem then $R_{n_k}$ is
strong sweeping out. We note that this assertion for Riemann sums
was proved by M. Akcoglu, A. Bellow, R. Jones, V. Losert, K.
Reinhold-Larsson and M. Wierdl in \cite{Akc} by using Rudin's
ideas. Now consider the operators
\begin{equation}\label{Khin}
\frac{1}{n}\sum_{j=1}^nf(jx).
\end{equation}
J. M. Marstrand in \cite{Mar}, solving Kinchine's conjecture, has
proved this sequence has $1$-sweeping out property. Applying
\trm{T6} we get the sequence \e{Khin} is strong sweeping out. We
note that alternate proofs of Rudin's and Marstrand's theorems
follows from Bourgain Entropy Theorem \cite{Bour1} a general tool
for investigation of divergence of certain operator sequences.
\begin{proof}[Proof of Rudin's theorem based on \trm{T4}]
Fix a number $0<\delta <1$. According to the conditions of Rudin's
theorem for any $k\in \ZN$ there exists a collection
$D_k=\{n_1,n_2,\ldots ,n_k\}\subset D$ such that no member of
$D_k$ divides the least common multiple of the others. It means we
can choose primes $p_{\nu_1},p_{\nu_2},\ldots ,p_{\nu_k}$ such
that $p_{\nu_j}|n_{\nu_j}$ and $p_{\nu_j}\not |n_{\nu_i}$ if
$i\neq j$. Let $l$ be the least common multiple of the numbers
$n_1,n_2,\ldots ,n_k$. Denoting $q_j=l/n_j$ we have
\begin{equation*}
l/D_k=\{q_1,q_2,\ldots,q_k\}.
\end{equation*}
In addition
\begin{equation*}
q_j=p_{\nu_1}^{m_1(j)}\cdot p_{\nu_2}^{m_2(j)}\ldots
p_{\nu_k}^{m_k(j)}\cdot\gamma_j
\end{equation*}
where
\begin{equation*}
m_j(j)=0,\quad m_i(j)>0,\hbox { if } i\neq j.
\end{equation*}
Denote by $Q_j$ the collection of rectangles \e{Rl} corresponding
to $l=q_j$ and suppose $Q=\cup_{j=1}^kQ_k$. According to \e{MBl}
we have
\begin{equation*}
\ZM_{l/D_k}f(x)=\sup_{B\in Q:\, x\in
B}\frac{1}{|B|}\int_B|f(t)|dt.
\end{equation*}
On the other hand any rectangle of the form
\begin{align*}
\left\{x\in\ZT^\infty:\,
\left[\frac{t_i}{p_{\nu_i}},\frac{t_i+1}{p_{\nu_i}}\right),\,
1\le j\le k,\, j\neq i\right\},\\
 0\le t_i<p_{\nu_i},\quad 1\le j\le k,\, j\neq i,
\end{align*}
can be represented as a disjoint union of rectangles from $Q_i$.
Thus the same assertion is true also for the set
\begin{equation*}
C_i=\{x\in\ZT^\infty :\, 0\le x_{\nu_j}< \frac{r_j}{p_{\nu(j)}},\,
1\le j\le k,\, j\neq i\},,\quad r_j=[\delta p_{\nu(j)}]+1.
\end{equation*}
Denote
\begin{equation*}
C=\bigcap_{j=1}^kC_j=\{x\in\ZT^\infty :\, 0\le x_{\nu_j}<
\frac{r_j}{p_{\nu(j)}},\, j=1,2,\ldots ,k\}.
\end{equation*}
It is easy to observe if $B\in Q_j$ and $B\subset C_j$ then
\begin{equation*}
|B\cap C|=\frac{r_j}{p_{\nu(j)}}|B|.
\end{equation*}
Therefore, since $\frac{r_j}{p_{\nu(j)}}>\delta $ we obtain
\begin{equation*}
\ZM_{l/D_k}\ZI_C(x)>\delta,\quad x\in \bigcup_{j=1}^kC_j
\end{equation*}
On the other hand we have
\begin{equation*}
\left|\bigcup_{j=1}^kC_j\right|=|C|
\left(1+\sum_{j=1}^k\frac{p_{\nu(j)}}{r_j}\right)>(k+1)|C|.
\end{equation*}
Thus we get
\begin{equation*}
|\{x\in\ZT^\infty:\, \ZM_{l/D_k}\ZI_C(x)>\delta\}|>(k+1)|C|
\end{equation*}
According \trm{T4} for some $G\subset \ZT $ we get
\begin{equation*}
|\{x\in\ZT :\, \zR_{D_k}^l\ZI_G(x)>\delta\}|>(k+1)|G|.
\end{equation*}
In addition, since $C$ is $\ZB_l$-measurable we have $G$ is
$\zI_l$-measurable. Thus from \e{DDl} we conclude
\begin{multline*}
|\{x\in\ZT :\, \zR_D\ZI_G(x)>\delta\}|\ge|\{x\in\ZT :\,
\zR_{D_k}\ZI_G(x)>\delta\}|\\
=|\{x\in\ZT :\,
\zR_{D_k}^l\ZI_G(x)>\delta\}|\ge(k+1)|G|,\quad k=1,2,\ldots.
\end{multline*}
This implies the sequence $R_nf(x)$, $n\in D$, has weak $\delta$
sweeping out property for any $0<\delta <1$. Applying \trm{T7} we
obtain it has strong sweeping out property. The proof is complete.
\end{proof}
\end{section}
\centerline{\footnotesize ACKNOWLEDGEMENT}

{\footnotesize This research work was kindly supported by College
of Science-Research Center Project No. Math/2008/07, Mathematics
Department, College of Science, King Saud University.}

\end{document}